\documentclass[11pt
]{amsart}

\usepackage{hyperref}
\hypersetup{bookmarksdepth=3}

\usepackage{geometry}
\usepackage{amsmath,amssymb}
\usepackage{
mathrsfs}
\usepackage{esint}

\usepackage{tensor}

\usepackage{amssymb,amsmath,amsthm}

\newtheorem{theorem}{Theorem}
\newtheorem{lemma}[theorem]{Lemma}

\newtheorem{definition}[theorem]{Definition}

\theoremstyle{remark}\newtheorem{remark}[theorem]{Remark}

\usepackage{graphicx}

\newcommand{\eq}[1]{\begin{align*}#1\end{align*}}  

\begin{document}

\title{A dyadic analysis approach to the problem of continuity of weighted estimates with respect to the $A_{p}$ characteristic.}
\author{Nikolaos Pattakos} 

\begin{abstract}
This paper presents a new proof of the results regarding the continuity of weighted estimates with respect to the characteristic of the weight. Here we first prove the result in the dyadic case which is ``easier" and then by the use of the Bellman function technique we pass to the continuous setting which is harder in general. To be more precise, as far as we know, this passage from the Martingale transform to the Hilbert transform described in this note is new.
\end{abstract}


\maketitle
\pagestyle {myheadings}

\begin{section}{introduction}
\markboth{\normalsize  Nikolaos Pattakos}{\normalsize  A Dyadic Analysis approach}
Throughout this article, we will denote $w$ as a positive $L^{1}_{loc}(\mathbb R^d)$ function which we call weight. For $1<p<\infty$ we say that $w$ belongs to the Muckenhoupt $A_{p}$ class if the quantity

\begin{equation}
\label{we}
[w]_{A_{p}}:=\sup_{Q}\Big(\frac1{|Q|}\int_{Q}w(x)\ dx\Big)\Big(\frac1{|Q|}\int_{Q}w(x)^{-\frac1{p-1}}\ dx\Big)^{p-1}
\end{equation}
is finite, where we take the supremum over all cubes in $\mathbb R^d$. These classes have been studied extensively over the last thirty years and many of their properties can be found in \cite{GCRF}. If we restrict ourselves in $\mathbb R$ and we consider dyadic intervals we can define the $A_{p}^{d}$ classes of weights in an obvious way. That is, we choose the cubes $Q$ to be dyadic intervals and the definition makes sense. In \cite{P1}, \cite{N}, \cite{NV1}, \cite{NV2} the following Theorem was proved.

\begin{theorem}
\label{TH}
Suppose that for some $1<p<+\infty$, a sublinear operator $T$ satisfies the inequality

$$\|T\|_{L^{p}(w)\rightarrow L^{p}(w)}\leq F([w]_{A_{p}}),$$
for all $w\in A_{p}$, where $F$ is a positive increasing function. Then

$$\lim_{[w]_{A_{p}}\to 1^+}\|T\|_{L^{p}(w)\rightarrow L^{p}(w)}=\|T\|_{L^{p}(dx)\rightarrow L^{p}(dx)},$$
and 

$$\|T\|_{L^{p}(w)\rightarrow L^{p}(w)}\leq\|T\|_{L^{p}(dx)\rightarrow L^{p}(dx)}(1+c\sqrt{\delta})$$
for all $A_{p}$ weights, $w$, with characteristic $[w]_{A_{p}}=1+\delta\approx 1$, and $c$ is a constant that depends on the dimension and $p$.
\end{theorem}
The proof used the Riesz-Thorin interpolation Theorem that appears in \cite{SW} and the nice interplay between the $A_{p}$ classes and the $BMO(\mathbb R^d)$ space. In this note, we try to derive this result using weaker machinery than the one mentioned before. In particular, we first prove the result for the Martingale transform using dyadic analysis and then by the use of the Bellman function technique we obtain the same result for the Hilbert transform on $\mathbb R$ for Poisson $A_{2}$ weights and the squares of the Riesz transforms on $\mathbb R^d$ for Heat $A_2$ weights. The purpose of this different proof is to present a different point of view in dealing with this kind of estimates and maybe give a better understanding of them.

Estimates like the one presented in Theorem \ref{TH} can be used in many different areas of mathematics. For example, in \cite{CS} such a continuity result was used for the study of PDE with random coefficients and in \cite{NIEL} the sharp asymptotic behavior of the $L^{2}(w)$ norm of the Riesz projection $P_{+}$, with respect to the $[w]_{A_{2}}$ characteristic, comes into play in the study of Schauder bases. Let us mention that Theorem \ref{TH} can be applied for a huge class of operators as the remarks dictate below.

\begin{remark}
In \cite{B2} Buckley showed that the Hardy-Littlewood maximal operator defined as

$$Mf(x):=\sup_{Q}\frac1{|Q|}\int_{Q}|f(y)|\ dy,$$
where the supremum is taken over all cubes in $\mathbb R^d$ that contain $x$, satisfies the estimate
$$\|M\|_{L^{p}(w)\rightarrow L^{p}(w)}\leq c[w]_{A_{p}}^{\frac1{p-1}},$$
for $1<p<+\infty$, and all weights $w\in A_{p}$, where the constant $c>0$ is independent of the weight $w$. This means that the assumptions of Theorem \ref{TH} hold for $M$.  
\end{remark}
\begin{remark}
Consider any Calder\'on-Zygmund operator $T$. By \cite{H1} we know that:
$$\|T\|_{L^{p}(w)\rightarrow L^{p}(w)}\leq c\,[w]_{A_{p}}^{max(1,\frac{1}{p-1})},$$ for any $A_{p}$ weight $w$, where $c>0$ is independent of the weight. This means that we can apply Theorem \ref{TH}, for $1<p<\infty$ and $F(x)=cx^{max(1,\frac{1}{p-1})}$.
\end{remark}  
The previous result was known as the $A_{2}$ conjecture because it suffices to prove it for $A_{2}$ weights and then you use a sharp extrapolation argument for the other values of $1<p<\infty$. Many mathematicians contributed in the proof of this very important Theorem and the first proof was done in \cite{H1}. The idea was to prove it for dyadic operators called the Haar shift operators and then by an ``averaging" argument to obtain the family of Calder\'on-Zygmund operators. For our estimates this averaging process can not work since many constants appear during the procedure and this is something that does not seem to be easily fixable. Very famous examples of such operators are the Hilbert transform

\begin{equation}
\label{H1}
Hf(x):=\frac{1}{\pi}\ p.v.\int_{\mathbb R}\frac{f(y)}{x-y}\ dy,
\end{equation}
and the Riesz transforms

\begin{equation}
\label{R1}
R_{j}f(x):=\frac{\Gamma(\frac{d+1}{2})}{\pi^{\frac{d+1}{2}}}\ p.v.\int_{\mathbb R^d}\frac{x_{j}-y_{j}}{|x-y|^{d+1}}\cdot f(y)\ dy,
\end{equation}
$1\leq j\leq d$. We will deal with these operators later in the paper and we will obtain the desired continuity for $H$ and $R_{j}^{2}$ using the Martingale transform that is defined as:

\begin{equation}
\label{Mart}
T_{\sigma}f=\sum_{I\in\mathcal D}\sigma_{I}(f,h_{I})h_{I},
\end{equation}
where $\mathcal D$ is a fixed dyadic lattice and $\sigma=\{\sigma_{I}\}_{I\in\mathcal D}\subseteq\{z\in\mathbb C:|z|\leq1\}$. Here we denote by $h_{I}$ the Haar function associated to the interval $I$ defined in the following way:

$$h_{I}(x)=\frac1{\sqrt{|I|}}\Big(\chi_{I_{+}}(x)-\chi_{I_{-}}(x)\Big).$$
By $I_{+}, I_{-}$ we mean the right half and the left half of the interval $I$, respectively. Note that the sequence $\{h_{I}\}_{I\in\mathcal D}$ is an orthonormal basis in $L^{2}$. Also, by $(f,h_{I})$ we mean the inner product of the two functions in $L^{2}$ and with $\mathcal D(J)$, where $J\in\mathcal D$ we denote all dyadic subintervals of $J$ including $J$ itself . Our goal is to present a new proof of the Theorem.

\begin{theorem}
\label{mainmain}
Suppose that $w$ is an $A_{2}^{d}$ weight with $[w]_{A_{2}^{d}}=1+\delta$ where $\delta\approx 0$. Then there is a universal constant $c>0$ such that

$$\|T_{\sigma}\|_{L^{2}(w)\rightarrow L^{2}(w)}\leq1+c\sqrt{\delta}.$$
\end{theorem}
In order to be able to derive the desired result we have to present some preliminary machinery that is needed. So we have the well known definition of a Carleson sequence.

\begin{definition}
\label{def}
A sequence of positive numbers $\{\alpha_{I}\}_{I\in\mathcal D}$ is called a Carleson sequence if there exists a constant $C>0$ such that for every dyadic interval $J$ we have 

\begin{equation}
\label{Car}
\frac1{|J|}\sum_{I\in\mathcal D(J)}\alpha_{I}\leq C.
\end{equation}
\end{definition}

This kind of sequences plays a very important role in dyadic Harmonic Analysis. It becomes clear if we take into consideration the Weighted Carleson Embedding Theorem that states:

\begin{theorem}
\label{Carcar}
Let $\{\alpha_{I}\}_{I\in\mathcal D}$ be a sequence of positive numbers and $w$ be a weight. Then,

$$\sum_{I\in\mathcal D}(fw^{\frac12})_{I}^{2}\alpha_{I}\leq 4C\|f\|_{2}^{2}$$
for all $f\in L^{2}(dx)$ if and only if 

$$\frac1{|J|}\sum_{I\in\mathcal D(J)}w_{I}^{2}\alpha_{I}\leq Cw_{J}$$
for all $J\in\mathcal D$.
\end{theorem}
This statement can be found in \cite{JW} and a proof in \cite{NTVV}. For a more general (but equivalent for constant weights) statement of Theorem \ref{Carcar} see \cite{JJ}. Finally, we need to define the Poisson class of $A_2$ weights and the Heat class of $A_2$ weights. For a weight $w$ on $\mathbb R^d$ we define its harmonic extension $w^{H}$ on $\mathbb R^{d+1}_{+}$ by the formula

$$w^{H}(x,t):=w\ast P_{t}(x)=\frac{\Gamma(\frac{d+1}{2})}{\pi^{\frac{d+1}{2}}}\int_{\mathbb R^{d}}w(y)\cdot\frac{t}{(t^{2}+|x-y|^{2})^{\frac{d+1}{2}}}\ dy,$$
where $P_{t}(x)$ is the Poisson Kernel, and its Poisson $A_2$ characteristic as

$$[w]_{A_{2}^{H}}:=\sup_{(x,t)\in\mathbb R^{d+1}_{+}}\Big[w^{H}(x,t)(w^{-1})^{H}(x,t)\Big].$$
The heat extension $w^h$ of $w$ on $\mathbb R^{d+1}_{+}$ is the function

$$w^{h}(x,t)=w\ast k_{t}(x)=\frac1{(4\pi t)^{\frac{d}{2}}}\int_{\mathbb R^d}w(y)\exp\Big(-\frac{|x-y|^{2}}{4t}\Big) \ dy,$$
where $k_{t}(x)$ is the heat kernel that is the fundamental solution of the equation $u_{t}-\Delta u=0$ and the Heat $A_2$ characteristic is given by

$$[w]_{A_{2}^{h}}:=\sup_{(x,t)\in\mathbb R^{d+1}_{+}}w^{h}(x,t)(w^{-1})^{h}(x,t).$$
For these two classes of weights many interesting facts are true. For instance, the Muckenhoupt $A_2$ weights and the heat $A_{2}$ weights are the same as sets and the two characteristics are comparable as it was proved in \cite{PV1}. For the Poisson $A_2$ class things are different in general. If $d>1$, there are power weights which are in the classical Muckenhoupt $A_2$ class for which the Poisson integral diverges and therefore, these power weights are not in the Poisson $A_{2}$ class as it was observed in \cite{P42}. In one dimension the numbers $[w]_{A_{2}^{H}}$ and $[w]_{A_{2}}$ are comparable and as a result the two sets of weights are equal.

In the next section we deal with our dyadic operators and in section \ref{sec3} we give the proof for the Hilbert transform by the use of the Bellman function technique. As it happens in most of the cases an estimate of the Martingale transform helps us to derive the existence of a very special function $B$ which in its turn yields to estimates of continuous operators.

\end{section}

\begin{section}{proof of theorem \ref{mainmain}}
Let us consider $f\in L^{2}(w)$ and write the duality

$$\|T_{\sigma}f\|_{L^{2}(w)}=\sup_{\|g\|_{L^{2}(w^{-1})}=1}\Big|\int T_{\sigma}f\cdot g\ dx\Big|=\sup_{\|g\|_{L^{2}(w^{-1})}=1}\Big|\int \sum_{I\in\mathcal D}\sigma_{I}(f,h_{I})h_{I}(x)g(x)\ dx\Big|.$$
Expanding the function $g=\sum_{J\in\mathcal D}(g,h_{J})h_{J}$ we have that

$$\|T_{\sigma}f\|_{L^{2}(w)}=\sup_{\|g\|_{L^{2}(w^{-1})}=1}\Big|\int\sum_{I,J\in\mathcal D}\sigma_{I}(f,h_{I})(g,h_{J})h_{I}(x)h_{J}(x)\ dx\Big|$$ 
and since the Haar functions are orthonormal we have that the right hand side is equal to

$$\sup_{\|g\|_{L^{2}(w^{-1})}=1}\Big|\sum_{I\in\mathcal D}\sigma_{I}(f,h_{I})(g,h_{I})\Big|.$$
Next we assume that $f,g\in L^{2}(dx)$ by taking the supremum of 

$$\sup_{\|g\|_{L^{2}(dx)}=1}\Big|\sum_{I\in\mathcal D}\sigma_{I}(fw^{-\frac12},h_{I})(gw^{\frac12},h_{I})\Big|,$$
and we substitute the Haar system by the orthogonal system of functions 

$$h_{I}^{w}(x)=\frac{h_{I}(x)+\gamma_{w}^{I}\cdot\chi_{I}(x)}{\delta_{w}^{I}},$$
where $\gamma_{w}^{I}=-\frac{c_{I}}{|I|}$, $\gamma_{w^{-1}}^{I}=-\frac{d_{I}}{|I|}$ and 

$$c_{I}=\sqrt{|I|}\cdot\frac{w_{I_{-}}-w_{I_{+}}}{2w_{I}},\ d_{I}=\sqrt{|I|}\cdot\frac{(w^{-1})_{I_{-}}-(w^{-1})_{I_{+}}}{2(w^{-1})_{I}}$$
and here by writing $w_{I}$ or $(w^{-1})_{I}$ we mean the average value of the function over the given interval. Finally,
for the $\delta_{w}^{I}$ numbers we have the relation

$$(\delta_{w}^{I})^{2}=w_{I}\cdot\Big(1-\frac{c_{I}^{2}}{|I|}\Big)=\frac{w_{I_{+}}w_{I_{-}}}{w_{I}}.$$
After all these changes of variables we arrive at the sum of four summands $\Big|\Sigma_{1}+\Sigma_{2}+\Sigma_{3}+\Sigma_{4}\Big|$ that we have to estimate from above, where 

\begin{equation}
\label{S1}
\Sigma_{1}=\sum_{I\in\mathcal D}\sigma_{I}(fw^{-\frac12},h_{I}^{w^{-1}})\delta_{w^{-1}}^{I}(gw^{\frac12},h_{I}^{w})\delta_{w}^{I}
\end{equation}

\begin{equation}
\label{S2}
\Sigma_{2}=-\sum_{I\in\mathcal D}\sigma_{I}(fw^{-\frac12},\chi_{I})\gamma_{w^{-1}}^{I}(gw^{\frac12},h_{I}^{w})\delta_{w}^{I}
\end{equation}

\begin{equation}
\label{S3}
\Sigma_{3}=-\sum_{I\in\mathcal D}\sigma_{I}(fw^{-\frac12},h_{I}^{w^{-1}})\delta_{w^{-1}}^{I}(gw^{\frac12},\chi_{I})\gamma_{w}^{I}
\end{equation}
and finally,

\begin{equation}
\label{S4}
\Sigma_{4}=\sum_{I\in\mathcal D}\sigma_{I}(fw^{-\frac12},\chi_{I})\gamma_{w^{-1}}^{I}(gw^{\frac12},\chi_{I})\gamma_{w}^{I}.
\end{equation}
In all of them we pass the absolute value inside and we start with $\Sigma_{1}$. Since the product $w_{I_{-}}w_{I_{+}}\leq w_{I}^{2}$ (same holds for $w^{-1}$ in place of $w$) we get an estimate from above by the expression

$$\sqrt{[w]_{A_{2}^{d}}}\cdot\sum_{I\in\mathcal D}|(fw^{-\frac12},h_{I}^{w^{-1}})(gw^{\frac12},h_{I}^{w})|,$$
and by the use of Cauchy-Schwartz this is less than or equal to

$$\sqrt{[w]_{A_{2}^{d}}}\cdot\Big(\sum_{I\in\mathcal D}|(fw^{-\frac12},h_{I}^{w^{-1}})|^{2}\Big)^{\frac12}\Big(\sum_{I\in\mathcal D}|(gw^{\frac12},h_{I}^{w})|^{2}\Big)^{\frac12},$$
which is equal to 

$$\sqrt{[w]_{A_{2}^{d}}}\cdot\Big(\sum_{I\in\mathcal D}|(fw^{\frac12},h_{I}^{w^{-1}})_{L^{2}(w^{-1})}|^{2}\Big)^{\frac12}\Big(\sum_{I\in\mathcal D}|(gw^{-\frac12},h_{I}^{w})_{L^{2}(w)}|^{2}\Big)^{\frac12}.$$
Since the functions $\{h_{I}^{w}\}_{I\in\mathcal D}$ are orthogonal in $L^{2}(w)$ (the same is true for the functions $\{h_{I}^{w^{-1}}\}_{I\in\mathcal D}$ in $L^{2}(w^{-1})$) we see that $\Sigma_{1}$ is bounded above by the quantity

$$\sqrt{[w]_{A_{2}^{d}}}\cdot\|f\|_{2}\|g\|_{2}.$$
Now in order to estimate the remaining three sums we need the Lemma.

\begin{lemma}
\label{USE}
For all weights $w$ whose characteristic is sufficiently close to $1$ there is an $\epsilon>0$ and a constant $C>0$ such that for all $J\in\mathcal D$ we have

$$\frac1{|J|}\sum_{I\in\mathcal D(J)}(w^{-1})_{I}^{2}\ d_{I}^{2}\ w_{I}\leq C\ (\log[w]_{A_{2}^{d}})^{\epsilon}\ [w]_{A_{2}^{d}}\ (w^{-1})_{J}$$
and

$$\frac1{|J|}\sum_{I\in\mathcal D(J)}(w_{I})^{2}\ c_{I}^{2}\ (w^{-1})_{I}\leq C\ (\log[w]_{A_{2}^{d}})^{\epsilon}\ [w]_{A_{2}^{d}}\ w_{J}.$$
\end{lemma}

\begin{proof}
Let us prove the second inequality and for the first one we follow the same steps.  In \cite{JW} it was proved that 

$$\frac{1}{|J|}\sum_{I\in\mathcal D(J)}c_{I}^{2}\leq 2\log [w]_{A_{2}^{d}},$$
which means that the sequence of numbers $\{c_{I}^{2}\}_{I\in\mathcal D}$ is a Carleson sequence and we are going to use this fact. Obviously,

$$\frac1{|J|}\sum_{I\in\mathcal D(J)}(w_{I})^{2}\ c_{I}^{2}\ (w^{-1})_{I}\leq\frac{[w]_{A_{2}^{d}}}{|J|}\sum_{I\in\mathcal D(J)}w_{I}\Big(\frac{w_{I_{-}}-w_{I_{+}}}{2w_{I}}\Big)^{2}|I|,$$
and this is equal to

$$\frac{[w]_{A_{2}^{d}}}{|J|}\sum_{I\in\mathcal D(J)}w_{I}\Big(\frac{w_{I_{-}}-w_{I_{+}}}{2w_{I}}\Big)^{2-2\epsilon}\Big(\frac{w_{I_{-}}-w_{I_{+}}}{2w_{I}}\Big)^{2\epsilon}|I|^{1-\epsilon}|I|^{\epsilon},$$
and by the use of H\"older's inequality for $p=\frac1{1-\epsilon}$ and $q=\frac1{\epsilon}$ (we will choose the number $\epsilon$ later)  we arrive at

$$\frac{[w]_{A_{2}^{d}}}{|J|}\Big(\sum_{I\in\mathcal D(J)}(w_{I})^{p}c_{I}^{2}\Big)^{\frac1{p}}\Big(\sum_{I\in\mathcal D(J)}c_{I}^{2}\Big)^{\frac1{q}}\leq\frac{[w]_{A_{2}^{d}}(2|J|\log[w]_{A_{2}^{d}})^{\frac1{q}}}{|J|}\Big(\sum_{I\in\mathcal D(J)}(w_{I})^{p}c_{I}^{2}\Big)^{\frac1{p}}.$$
Now since the sequence $\{c_{I}^{2}\}_{I\in\mathcal D}$ is Carleson we can bound the last expression (see \cite{JJ}) by

$$\frac{C}{|J|}[w]_{A_{2}^{d}}(|J|\log[w]_{A_{2}^{d}})^{\frac1{q}}|J|^{\frac1{p}}(w^{p})_{I}^{\frac1{p}}=C[w]_{A_{2}^{d}}(\log[w]_{A_{2}^{d}})^{\frac1{q}}(w^{p})_{I}^{\frac1{p}}.$$
It is time to use the Reverse H\"older inequality with the exponent $p$ (see \cite{GCRF}) and it yields the upper bound

$$C(\log[w]_{A_{2}^{d}})^{\epsilon}[w]_{A_{2}^{d}}w_{J}.$$
So far the number $\epsilon$ depends on the weight $w$ but as it was observed in \cite{P1} for weights $w$ that are close to the constant weight $1$, in the sense that the number $[w]_{A_{2}^{d}}$ is close to $1$ we can choose $\epsilon$ independently of $w$ and the proof is complete.
\end{proof}

Let us see what is happening with sums $\Sigma_{2}$ and $\Sigma_{3}$. They are similar so we will handle $\Sigma_{2}$ only. Again we put the absolute value inside and using the fact that $1-\frac{c_{I}^{2}}{|I|}\leq 1$ and applying the Cauchy-Schwartz once again we obtain the following bound from above

$$\Big(\sum_{I\in\mathcal D}(fw^{-\frac12})_{I}^{2}d_{I}^{2}w_{I}\Big)^{\frac12}\Big(\sum_{I\in\mathcal D}(gw^{\frac12},h_{I}^{w})^{2}\Big)^{\frac12}.$$
The second of these terms is bounded by the $L^{2}(dx)$ norm of $g$ as it happened in $\Sigma_{1}$. We estimate the first sum by the use of the Carleson Embedding Theorem \ref{Carcar} and Lemma \ref{USE} to arrive at the upper bound

$$C\sqrt{(\log[w]_{A_{2}^{d}})^{\epsilon}[w]_{A_{2}^{d}}}\|f\|_{2}\|g\|_{2}.$$
Finally, our last sum, $\Sigma_{4}$, can be treated in the same way. We apply Cauchy-Schwartz and get

$$ \Big(\sum_{I\in\mathcal D}(fw^{-\frac12})_{I}^{2}d_{I}^{2}w_{I}\Big)^{\frac12}\Big(\sum_{I\in\mathcal D}(gw^{\frac12})_{I}^{2}c_{I}^{2}(w^{-1})_{I}\Big)^{\frac12}$$
and again by the Carleson Embedding Theorem and Lemma \ref{USE} we see that $|\Sigma_{4}|$ is less than or equal to

$$C(\log[w]_{A_{2}^{d}})^{\epsilon}[w]_{A_{2}^{d}}\|f\|_{2}\|g\|_{2},$$
which finishes the proof of Theorem \ref{mainmain}.

\end{section}

\begin{section}{The continuous setting}
\label{sec3}
Our main goal in this section is to pass from the dyadic setting, that is the Martingale Transform, to the continuous setting that is the Hilbert and the squares of the Riesz transforms in this case. This point of view gives us a lot of information about the interplay between the dyadic operators and the continuous operators. This connection is actually one of the most interesting things in todays Harmonic Analysis. The dyadic operators have always been the prototypes of the continuous, more ``complicated", ones. We are going to prove the following Theorem regarding the Hilbert transform.

\begin{theorem}
\label{HIL}
Suppose that $w$ is an $A_{2}^{H}$ weight with $[w]_{A_{2}^{H}}=1+\delta$ where $\delta\approx 0$. Then there is a universal constant $c>0$ such that

$$\|H\|_{L^{2}(w)\rightarrow L^{2}(w)}\leq1+c\sqrt{\delta}.$$
\end{theorem} 

\begin{proof}
We start by rewriting the conclusion of Theorem \ref{mainmain} which by using duality yields to

\begin{eqnarray}\label{bellmansetup}
&&\frac1{4|J|}\sum_{I\in\mathcal D(J)}|\langle f\rangle_{I_{+}}-\langle f\rangle_{I_{-}}||\langle g\rangle_{I_{+}}-\langle g\rangle_{I_{-}}||I|\\\nonumber
&\leq& (1+c\sqrt{\delta})\langle |f|^{2}w\rangle_{J}^{1/2}\langle |g|^{2}w^{-1}\rangle_{J}^{1/2},
\end{eqnarray}
for any $J\in\mathcal D$, any $f\in L^{2}(w)$ and $g\in L^{2}(w^{-1})$. The definition of the Bellman function is given as the supremum of the left hand side of (\ref{bellmansetup}):

$$B(X,Y,x,y,r,s)=\sup\Big\{\frac1{4|J|}\sum_{I\in\mathcal D(J)}|\langle f\rangle_{I_{+}}-\langle f\rangle_{I_{-}}||\langle g\rangle_{I_{+}}-\langle g\rangle_{I_{-}}||I|:$$
$$\langle f\rangle_{J}=x, \langle g\rangle_{J}=y, \langle w\rangle_{J}=r, \langle w^{-1}\rangle_{J}=s,\langle |f|^{2}w\rangle_{J}=X, \langle |g|^{2}w^{-1}\rangle_{J}=Y\Big\}.$$
We are going to denote the domain by:

$$D=\{0<(X,Y,x,y,r,s)\in\mathbb R^{6} :|x|^{2}<Xs, |y|^{2}<Yr, 1<rs<Q\},$$
where $Q=1+\delta$. This function has two very useful properties. One of them is that it satisfies the inequality:

$$0\leq B(X,Y,x,y,r,s)\leq(1+c\sqrt{\delta})X^{\frac12}Y^{\frac12},$$
for all 6-tuples of our domain, and also for all vectors $v, v^{+}, v^{-}\in D$ such that $v=\frac{v^{+}+v^{-}}{2}$ we have the inequality:

$$B(v)-\frac{B(v^{+})+B(v^{-})}{2}\geq\frac14|x^{+}-x^{-}||y^{+}-y^{-}|,$$
where $v=(X^{+}, Y^{+}, x^{+}, y^{+}, r^{+}, s^{+})$ and similarly for $v^{-}$. The proof of the second inequality is standard for the Bellman function technique and we refer the interested reader to \cite{PV1}. Assume now that a compact subset $K$ of the interior of $D$ is given. By choosing $0<\epsilon<<dist\{K,\partial D\}$, and by considering a $C^{\infty}_{c}$, Bell shaped function $S$, supported in the unit ball of $\mathbb R^{6}$ with center $0$ we can molify our function $B$ in order to get a smooth version of it, in a small neighborhood of $K$ in the following way:

$$B_{Q, K}:=B\ast S_{\epsilon},$$
where $S_{\epsilon}(x)=\frac{1}{\epsilon^{6}}S(\frac{x}{\epsilon})$. We can check that this new function inherits some of the properties of $B$ in the following sense. The Hessian matrix of $B_{Q, K}$ is non-positive definite i.e.:

\begin{equation}
\label{eq2}
\Big(-d^{2}B_{Q, K}v, v\Big)_{\mathbb R^{6}}\geq 2|v_{3}v_{4}|,
\end{equation}
and for all $(X,Y,,x,y,r,s)\in D$:

\begin{equation}
\label{eq3}
0\leq B_{Q,K}(X,Y,x,y,r,s)\leq(1+c\sqrt{\delta})(1+c_{K}\epsilon)X^{\frac12}Y^{\frac12},
\end{equation}
where $c_{K}$ is a constant that depends on the distance of our compact set $K$ to the boundary of $D$. We will use this smooth function $B_{Q,K}$ to obtain our continuity result. For this we need to consider functions $\phi,\psi\in C^{\infty}_{c}(\mathbb R)$, and observe that:

\begin{equation}
\label{eq4}
\Big|\int_{\mathbb R}H(\phi)\psi dx\Big|\leq\int_{\mathbb R}\int_{0}^{\infty}|\nabla\phi^{H}||\nabla\psi^{H}|tdtdx,
\end{equation}
where the functions appearing on the right hand side are the harmonic extensions to the upper half plane of the $\phi, \psi$ respectively. Let us justify inequality (\ref{eq4}) for the unit disk in the complex plane. Using Green's Theorem we obtain:

\begin{eqnarray*}
\int_{\mathbb S^{1}}\alpha\beta d\theta=\int_{\mathbb S^{1}}\alpha\cdot\beta\cdot\frac{\partial G}{\partial n}d\theta-\int_{\mathbb S^{1}}\frac{\partial(\alpha\beta)}{\partial n}Gd\theta&=&\int_{B}\alpha\beta\Delta Gdxdy-\int_{B}\Delta(\alpha\beta)Gdxdy\\
&=&(\alpha\beta)(0)-\int_{B}\nabla\alpha\cdot\nabla\beta\cdot Gdxdy
\end{eqnarray*}
and for $\alpha=H\phi$ and $\beta=\psi$ we arrive at inequality \ref{eq4}, since $|\nabla H\phi|=|\nabla\phi|$. Note that we use the notation $B$ for the unit disk in the complex plane, and by $G$ we denote Green's function in $B$. From now on the functions $\phi,\psi$ will be fixed. We also fix a weight $w$ in the Poisson $A_{2}$ class such that $Q_{A_{2}}(w)\leq1+\delta:=Q$, where $Q$ is the number that appears in the definition of $D$. Define the function 
$$v(x,t)=((\phi^{2}w)^{H}, (\psi^{2}w^{-1})^{H}, \phi^{H}, \psi^{H}, w^{H}, (w^{-1})^{H}),$$
that maps $\mathbb R^{2}_{+}$ into the domain $D$ (more precisely it maps compact subsets of the interior of $\mathbb R^{2}_{+}$ into compact subsets of the interior of $D$). To see this it helps to know that $A_{2}$ weights have at most logarithmic singularities. Fix any compact subset $M$ of the upper half plane and a rectangle $K$ of $\mathbb R^{2}_{+}$, say $M\subset K=[-r,r]\times [\frac1{\rho}, R]$. Then $v(K)$ is compact in $D$ and we can find a function $B=B_{Q, v(K)}$ infinitely differentiable in a small neighborhood of $v(K)$ such that inequalities $(1),(2)$ are fulfilled.  Now we define the function:

$$b(x,t)=B(v(x,t)),$$
for $(x,t)\in K$, and observe that: 

$$\Big(\frac{\partial^{2}}{\partial t^{2}}+\frac{\partial^{2}}{\partial x^{2}}\Big)b(x,t)=\Big(d^{2}B(v(x,t))\frac{\partial v(x,t)}{\partial t},\frac{\partial v(x,t)}{\partial t}\Big)_{\mathbb R^{6}}+\Big(d^{2}B(v(x,t))\frac{\partial v(x,t)}{\partial x},\frac{\partial v(x,t)}{\partial x}\Big)_{\mathbb R^{6}}$$
since all the entries of $v$ are harmonic extensions and this implies $(\nabla B(v(x,t)),(\frac{\partial^{2}}{\partial t^{2}}+\frac{\partial^{2}}{\partial x^{2}})v(x,t))_{\mathbb R^{6}}=0$. We will assume that the following calculations are correct, and after we finish the proof, we will make this more precise in remark \ref{rem7}. We define the function:

$$\beta(t)=\int_{-r}^{r}b(x,t)dx$$
for $t\in[\frac{1}{\rho},R]$. Then:

\begin{eqnarray*}
\beta(R)-\beta(\frac{1}{\rho})&=&\int_{\frac1{\rho}}^{R}\frac{d}{dt}\beta(t)dt\\
&=&\int_{-r}^{r}\int_{\frac1{\rho}}^{R}\frac{d}{dt}b(x,t)dtdx\\
&=&-\int_{-r}^{r}\int_{\frac1{\rho}}^{R}\frac{d^2}{dt^{2}}b(x,t)tdtdx+Q_{1}(R,r,\rho)\\
&=&-\int_{-r}^{r}\int_{\frac1{\rho}}^{R}\Big(\frac{d^{2}}{dt^{2}}+\frac{d^2}{dx^{2}}\Big)b(x,t)tdtdx+Q_{2}(R,r,\rho) 
\end{eqnarray*}
where $Q_{1}(R,r,\rho)=-b(r,R)R+b(-r,R)R+b(r,\frac{1}{\rho})\frac1{\rho}-b(-r,\frac1{\rho})\frac1{\rho}$ and $Q_{2}(R,r,\frac1{\rho})=Q_{1}(R,r,\frac1{\rho})+\int_{-r}^{r}\int_{\frac1{\rho}}^{R}\frac{d^{2}}{dx^{2}}b(x,t)tdtdx$. The last integral, in the series of equalities, is equal to:

\begin{equation}
\label{littleb}
\int_{-r}^{r}\int_{\frac1{\rho}}^{R}\Big[\Big(-d^{2}B(v(x,t))\frac{\partial v(x,t)}{\partial t},\frac{\partial v(x,t)}{\partial t}\Big)_{\mathbb R^{6}}+\Big(-d^{2}B(v(x,t))\frac{\partial v(x,t)}{\partial x},\frac{\partial v(x,t)}{\partial x}\Big)_{\mathbb R^{6}}\Big]tdtdx.
\end{equation}
The following is a Lemma proved in \cite{DV}:

\begin{lemma}
Let $m, n, k\in\mathbb N$. Denote by $d=m+n+k$. For arbitrary $v\in\mathbb R^d$ write $v=v_{m}\oplus v_{n} \oplus v_{k}$, where $v_{i}\in\mathbb R^{i}$ for $i=n, m, k$. Let $R=\|v_{m}\|, r=\|v_{n}\|$. Suppose that a matrix $A\in\mathbb R^{d,d}$ is such that:

$$\langle Av,v\rangle\geq2Rr,$$
for all $v\in\mathbb R^d$. Then, there is $\tau>0$, satisfying:

$$\langle Av,v\rangle\geq\tau R^{2}+\frac1{\tau}r^{2},$$
again for all $v\in\mathbb R^d$. 

\end{lemma}

Using inequality (\ref{eq2}) and the previous Lemma, we get that for every $(x,t)\in K$ there is $\tau=\tau(x,t)>0$ with the property:

$$\Big(-d^{2}B(v(x,t))\frac{\partial v(x,t)}{\partial t},\frac{\partial v(x,t)}{\partial t}\Big)_{\mathbb R^{6}}\geq\tau\Big|\frac{\partial\phi^{h}}{\partial t}\Big|^{2}+\frac{1}{\tau}\Big|\frac{\partial\psi^{h}}{\partial t}\Big|^{2}$$
(the same is true for the $x$ variable in the place of the $t$ variable). This implies that the integral in (\ref{littleb}) is at least:

$$\int_{-r}^{r}\int_{\frac1{\rho}}^{R}\Big[\tau\Big|\frac{\partial\phi^{h}}{\partial t}\Big|^{2}+\frac{1}{\tau}\Big|\frac{\partial\psi^{h}}{\partial t}\Big|^{2}+\tau\Big|\frac{\partial\phi^{h}}{\partial x}\Big|^{2}+\frac1{\tau}\Big|\frac{\partial\psi^{h}}{\partial x}\Big|^{2}\Big]tdtdx$$
and using the well known arithmetic mean--geometric mean inequality we obtain:

$$\beta(R)-\beta(\frac1{\rho})-Q_{2}(R,r,\frac{1}{\rho})\geq\int_{-r}^{r}\int_{\frac1{\rho}}^{R}\Big(\Big|\frac{\partial\phi^{h}}{\partial x}\Big|^{2}+\Big|\frac{\partial\phi^{h}}{\partial t}\Big|^{2}\Big)^{\frac12}\Big(\Big|\frac{\partial\psi^{h}}{\partial x}\Big|^{2}+\Big|\frac{\partial\psi^{h}}{\partial t}\Big|^{2}\Big)^{\frac12}tdtdx.$$ 
The last double integral is obviously bigger than or equal to the same double integral over the compact subset of $\mathbb R^{2}_{+}$, $M$. By inequality (\ref{eq3}) we have:

$$\beta(R)=\int_{-r}^{r}b(x,R)dx\leq(1+\epsilon)(1+c\sqrt{\delta})\int_{\mathbb R}\Big((\phi^{2}w)^{h}(x,R)(\psi^{2}w^{-1})^{h}(x,R)\Big)^{\frac12}dx,$$
and by H\"older's inequality, we arrive at:

\begin{eqnarray*}
\beta(R)&\leq&(1+\epsilon)(1+c\sqrt{\delta})\Big(\int_{\mathbb R}\phi(x)^{2}w(x)dx\Big)^{\frac12}\Big(\int_{\mathbb R}\psi^{2}(x)w^{-1}(x)dx\Big)^{\frac12}\\
&=&(1+\epsilon)(1+c\sqrt{\delta})\|\phi\|_{L^{2}(w)}\|\psi\|_{L^{2}(w^{-1})}.
\end{eqnarray*}
Finally, letting $\rho\to\infty$ and then $r\to\infty$ we see that: 

$$\lim_{r\to\infty}\lim_{\rho\to\infty}Q_{2}(R,r,\rho)=0,$$  
and so:

$$\beta(R)\geq\beta(R)-\beta(0)\geq\int\int_{M}\Big(\Big|\frac{\partial\phi^{h}}{\partial x}\Big|^{2}+\Big|\frac{\partial \phi^{h}}{\partial t}\Big|^{2}\Big)^{\frac12}\Big(\Big|\frac{\partial \psi^{h}}{\partial x}\Big|^{2}+\Big|\frac{\partial\psi^{h}}{\partial t}\Big|^{2}\Big)^{\frac12}tdtdx.$$
Combining everything together:

$$(1+\epsilon)(1+c\gamma(\delta))\|\phi\|_{L^{2}(w)}\|\psi\|_{L^{2}(w^{-1})}\geq\int\int_{M}\Big(\Big|\frac{\partial\phi^{h}}{\partial x}\Big|^{2}+\Big|\frac{\partial \phi^{h}}{\partial t}\Big|^{2}\Big)^{\frac12}\Big(\Big|\frac{\partial \psi^{h}}{\partial x}\Big|^{2}+\Big|\frac{\partial\psi^{h}}{\partial t}\Big|^{2}\Big)^{\frac12}tdtdx,$$
for an arbitrary compact subset $M$ of $\mathbb R^{2}_{+}$. Letting $M$ expand to the whole upper half plane, we get exactly what we need. Using (\ref{eq4}) we obtain:

$$\Big|\int_{\mathbb R}H(\phi)\psi dx\Big|\leq(1+\epsilon)(1+c\sqrt{\delta})\|\phi\|_{L^{2}(w)}\|\psi\|_{L^{2}(w^{-1})}.$$
This is true for all weights $w$ in the Poisson $A_{2}$ class with $[w]_{A_{2}^{H}}=1+\delta$ and the proof is complete.
\end{proof}

\begin{remark}
\label{rem7}
This is to explain why all of the previous manipulations of the Bellman function work, and write some comments about the proof.  For the calculations we need to look at how the Bellman function that we used was constructed. The function $B_{Q,v(K)}$ was the convolution of a positive $L^{\infty}_{loc}(\mathbb R^{6})$ function and a positive, bounded $C^{\infty}_{c}(\mathbb R^{6})$ function supported in a really small $\epsilon$-ball of $\mathbb R^{6}$ around $0$. This means that $B_{Q, v(K)}$ is in $C^{\infty}\cap L^\infty.$ In addition, note that for a fixed value of $t>0$, the functions $v(x,t)$ and $\nabla v(x,t)$ go to $0$, as $x\to\infty$, rather fast. Also if needed, we can assume that our nonconstant weight $w$, is smooth and constant outside some large ball, as in \cite{PV1}.
\end{remark}

\begin{remark}
The proof of Theorem \ref{HIL}  has an interesting feature. The Hilbert transform is not in the closure of convex combinations of Martingale transforms. It is an ``average" of dyadic shift operators in the sense described in \cite{PP}. But philosophically knowing an estimate for the Martingale transform gives us a similar one for the Hilbert transform and many other continuous operators. 
\end{remark}

For the sake of completeness we can prove a very similar result to Theorem \ref{HIL} for the squares of the Riesz transforms but instead of the Poisson $A_{2}$ characteristic of the weight we have the Heat $A_{2}$ characteristic. This was proved again in \cite{Bor} for matrix valued weights and for this reason we only mention the result.

\begin{theorem}
\label{RIE}
There is a constant $c>0$ such that for all weights $w$ in $\mathbb R^d$ of Heat $A_{2}$ characteristic sufficiently close to $1$ the estimate
\begin{equation}
\label{Riesz}
\Big\|\sum_{i=1}^{n}\sigma_{j_{i}}R_{j_{i}}^{2}\Big\|_{L^{2}(w)\rightarrow L^{2}(w)}\leq1+c\cdot\sqrt{[w]_{A_{2}^{h}}-1}
\end{equation}
is satisfied, where $\{ j_{1},j_{2}, \dots,j_{n}\}$ is an arbitrary subset of $\{1, \dots, d\}$ and $\sigma=\{\sigma_{j_{i}}\}_{i=1}^{n}$ is an arbitrary choice of signs.
\end{theorem}
The Lemma that allows us to connect the squares of the Riesz transforms and heat extensions is the following (see \cite{PV1}).

\begin{lemma}
Let $\phi, \psi$ be real functions in $C_{c}^{\infty}(\mathbb R^d)$. Then the integral $\int \frac{\partial \phi^{h}}{\partial x_{i}}\cdot\frac{\partial \psi^{h}}{\partial x_{i}} \ dxdt$ converges absolutely for all $i=1, \dots,d$ and 
\eq{\int_{\mathbb R^d} R_{i}^{2}\phi\cdot\psi\ dx=-2\int_{\mathbb R^{d+1}_{+}} \frac{\partial \phi^{h}}{\partial x_{i}}\cdot\frac{\partial \psi^{h}}{\partial x_{i}} \ dx\ dt.}
\end{lemma}
\textbf{Acknowledgments}: The author would like to thank prof. Alexander Volberg from Michigan State University since many of these ideas arose from discussions with him during his studies in MSU.

\end{section}


\vskip10pt

\noindent N. Pattakos,
{\it e-mail address}:  \texttt{nikolaos.pattakos@gmail.com}

\end{document}